\documentclass[12pt]{article}
\usepackage{amsmath}
\usepackage{graphicx,psfrag,epsf}
\usepackage{enumerate}
\usepackage{natbib}
\usepackage{amsmath,amsthm,amsfonts,epsfig,amssymb,natbib,eucal}
\usepackage{booktabs}
\usepackage{multirow}
\usepackage{siunitx}
\usepackage{qtree}
\usepackage{setspace}

\usepackage{float}
\restylefloat{table}
\usepackage{color}
\usepackage{url}
\usepackage{hyperref}

\usepackage{subcaption}

\usepackage{tikz}
\usepackage{forest}
\usetikzlibrary{trees}

\newtheorem{thm}{Theorem}
\newtheorem{lem}{Lemma}

\newtheorem{definition}{Definition}

\newtheorem{comment}{Comment}
\newtheorem{remark}{Remark}
\newtheorem{cj}{Conjecture}

\newcommand{\blind}{0}

\begin{document}

\def\spacingset#1{\renewcommand{\baselinestretch}%
{#1}\small\normalsize} \spacingset{1}


\if0\blind
{\title{\bf The Optimality of a Nested Generalized Pairwise Group Testing Procedure}

\author{Yaakov Malinovsky
\thanks{Department of Mathematics and Statistics, University of Maryland,
Baltimore County, Baltimore, MD 21250, USA.
Email: {\tt yaakovm@umbc.edu}.  Research supported in part
by BSF grant 2020063.}
\and
Viktor Skorniakov
\thanks{Institute of Applied Mathematics, Faculty of Mathematics and Informatics,
Vilnius University, Naugarduko 24, Vilnius LT-03225, Lithuania.
Email: {\tt e-mail: viktor.skorniakov@mif.vu.lt}.  }
}
 \maketitle
} \fi
\if1\blind
{
  \bigskip
  \bigskip
  \bigskip
  \begin{center}
    {\LARGE\bf }
\end{center}
  \medskip
} \fi

\bigskip
\begin{abstract}
We study the problem of identifying defective units in a finite population of \( n \) units, where each unit \( i \) is independently defective with known probability \( p_i \). This setting is referred to as the \emph{Generalized Group Testing Problem}. A testing procedure is called optimal if it minimizes the expected number of tests.
It has been conjectured that, when all probabilities \( p_i \) lie within the interval \( \left[1 - \frac{1}{\sqrt{2}},\, \frac{3 - \sqrt{5}}{2} \right] \), the \emph{generalized pairwise testing {algorithm}}, applied to the \( p_i \) arranged in nondecreasing order, constitutes the optimal nested testing strategy among all such order-preserving nested strategies.
In this work, we confirm this conjecture and establish the optimality of the procedure within the specified regime. Additionally, we provide a complete structural characterization of the procedure and derive a closed-form expression for its expected number of tests. These results offer new insights into the theory of optimal nested strategies in generalized group testing.
\end{abstract}

\noindent%
\medskip

 \vfill


\newpage
\spacingset{1.45} 
\section{Introduction }

\subsection{Overview}
The concept of group testing was introduced by Robert Dorfman in 1943 in the context of screening large numbers of U.S. Army recruits for syphilis during World War II \citep{D1943}.
His idea was to reduce the expected number of tests by pooling samples and performing a single combined test—a strategy that proved to be both elegant and efficient.
Historical details surrounding the formulation of the problem are nicely described in \cite{Dh1999}. Since then, group testing has developed into a substantial area of research with both theoretical depth and practical
relevance. In particular, the COVID-19 pandemic brought renewed attention to group testing methods as tools for large-scale, efficient population screening. A comprehensive overview of group testing methodologies under various assumptions and across diverse applications can be found, for example, in \cite{K1973, Dh1999, AJS2019, MA2019}.

In general, group testing models fall into two broad categories: probabilistic and combinatorial. Each of these frameworks can be studied in adaptive or non-adaptive settings,
depending on whether tests are designed sequentially or in advance. Additionally, both settings may allow for error-free or noisy testing, leading to a rich family of models.

The classical model introduced by Dorfman is probabilistic, error-free, and adaptive. This is also the framework we adopt in the present work.
Other variants of group testing—including noisy, combinatorial (and their connections to probabilistic models), as well as non-adaptive frameworks—are thoroughly discussed in \cite{K1973, Dh1999, Dh2006, BM2017, AJS2019, P2021, NZ2022, N2023, HO2025}, among others.

\subsection{Homogeneous setting}
Consider a set $S$ of $n$ units, where each unit is defective with probability $p$, and good with probability $q=1-p$, independently of the others. Following standard notation in the group testing literature, such a set is referred to as a binomial set \citep{SG1959}. A group test applied to a subset of size $n^{\prime}$ is a binary test with two possible outcomes: positive or negative. The outcome is negative if all $n^{\prime}$ units are good, and positive if at least one unit in the group is defective. We refer to such a subset as defective or contaminated.

Dorfman introduced the procedure by partitioning a large population into subsets of equal size and testing each subset as a whole. If the test outcome for a given subset $S^{\prime}$
is negative, all units within $S^{\prime}$
are declared to be good. Otherwise, each unit in
$S^{\prime}$ is subsequently tested individually.

The goal is the complete identification of all $n$ units with the minimum expected number of tests. To date, an optimal group testing procedure under the binomial model is unknown for general $n$ when $\displaystyle p<(3-\sqrt{5})/2$.
For $\displaystyle p\geq(3-\sqrt{5})/2$ (or equivalently, $3-q-q^2\geq 2$, with $q=1-p$), \cite{U1960} showed that individual (one-by-one) testing is optimal. At the boundary point, this procedure is an optimal, since in this case, individual testing and pairwise testing (defined below) have the same expected number of tests \citep{YH1990}.

Every reasonable group testing algorithm should satisfy the following two properties \citep{SG1959, U1960}: (1) units that are classified as positive or negative are never tested again, and (2) a test is not performed if its outcome can be inferred from previous test results. The Dorfman procedure does not satisfy this property, as it still tests the final individual even when all others in the group have tested negative and the group test is positive. In contrast, the modified Dorfman procedure avoids this redundancy by not testing the last individual in such cases \citep{SG1959}.

A nested algorithm is a group testing procedure with the property that, once a positive subset $I$ is identified, the next subset $I_1$
to be tested is always a proper subset of $I$; that is, $I_1\subset I.$ This natural class of group testing procedures was first established  by \cite{SG1959} and further developed in \cite{S1960}.
The first optimal nested procedure was proposed in \cite{SG1959}, employing backward recursion (dynamic programming) with computational complexity of order $n^3$.
Clearly, both the Dorfman procedure and the modified Dorfman procedure belong to the class of nested group testing procedures.
Substantial efforts have since been devoted to reducing this computational complexity, resulting in numerous publications. A review of these developments can be found in \cite{MA2019, MA2021}.

The pairwise nested algorithm, a member of the class of nested procedures, was introduced by \cite{YH1990}. It is defined as follows:
\smallskip

{\it
\begin{itemize}
\item[(i)]
If no contaminated set exists, then always test a pair from the binomial set
unless only one unit is left, in which case we test that unit.
\item[(ii)]
If a contaminated pair is found, test one unit of that pair. If that unit is
good, we deduce the other is defective. Thus we classify both units and only a
binomial set remains to be classified. If the tested unit is defective, then by a result
of \cite{SG1959}, the other unit together with the remaining binomial set
forms a new binomial set. So, both cases reduce to a binomial set.
It is easily verified that at all times the unclassified units belong to either a
binomial set or, a contaminated pair. Thus the pairwise testing algorithm is well
defined and is nested.
\end{itemize}
}
The following theorem provides a closed-form characterization of the optimal nested procedure, requiring no computational effort for its design.
\begin{thm}{\bf \cite{YH1990}}\\
\label{re:YH}
The pairwise testing algorithm is the unique (up to the substitution of
equivalent units) optimal nested algorithm for all $n$ if and only if $1-1/\sqrt{2}\leq p \leq (3-\sqrt{5})/2$ (at the boundary values the pairwise testing algorithm is an optimal
nested algorithm).
\end{thm}

Moreover, \cite{YH1990} derived the expected number of tests required by the pairwise testing algorithm. More recently, \cite{CS2024} established the higher-order moments of the total number of tests under this algorithm, as well as its asymptotic limiting distribution under suitable normalization.

\subsection{Generalized Group Testing and problem description}
Consider a set of \( n \) stochastically independent units \( u_1, u_2, \ldots, u_n \), where each unit \( u_i \) has a known probability \( p_i \in (0,1) \) of being defective, and a corresponding probability \( q_i = 1 - p_i \) of being good.
The objective is to determine the status (good or defective) of every unit using group testing. The order in which units are tested may be strategically chosen based on the known defectivity probabilities. The goal is to design a testing procedure that classifies all units as either good or defective while minimizing the total expected number of tests.

The generalized group testing problem was first introduced by \citet[pg.~144]{S1960}. In this work, the author considered scenarios where two (or more) different types of units, each with potentially different probabilities of being good, may be grouped together for testing. Specifically, in the case of two units with known probabilities $q_1\geq q_2$
of being good, individual testing is shown to be optimal if the inequality $3-q_1-q_1q_2>2$ holds. This condition emerges from the encoding algorithm construction originally developed by \citet{H1952} in the context of information theory, as applied to the case $n=2$ in \citet{S1960}.

However, despite more than 60 years having passed since the foundational work of \citet{S1960}, and subsequent advances in the field, the optimal solution—specifically, the optimal ordering of units that minimizes the expected number of tests—remains unknown, even for certain restricted procedures.
For Dorfman's procedure, \citet{H1975} showed that an optimal partition is an ordered partition, meaning that for any two subsets in the partition, all units in one subset have probabilities (of being good) greater than or equal to those in the other. This property allowed Hwang to find the optimal partition using a dynamic programming approach with computational complexity proportional to $n^2$.
However, for the modified Dorfman procedure (as defined above), this ordered partition property does not hold, as shown in \citet{M2017}. Moreover, a brute-force search for the best partition is infeasible, since the number of possible partitions grows exponentially with $n$, according to the Bell number.

\cite{KS1988} presented a dynamic programming algorithm with computational complexity proportional to $n^3$ for finding an optimal nested procedure, assuming a fixed order of units $u_1,\ldots,u_n$, which must be preserved throughout the testing process. Additionally, \cite{KS1988} employed the method of \cite{U1960} and extended the result of \cite{S1960} from the case $n=2$ to general $n$.
Specifically, they proved that if $3-q_1-q_1q_2>2$, where $q_1\geq\cdots \geq q_n$, then individual testing is optimal.
Closely related results were also obtained by \cite{YH1988b}.
In a recent work, \cite{SV2025} reduced the computational complexity of the dynamic programming algorithm to $O(n^2)$.

Before we proceed, we provide a few formal definitions applicable to the Generalized Group Testing Nested Procedure, in addition to the nested procedure defined earlier.

\begin{definition}[Ordered nested procedure] Let $\left(u_1,\ldots,u_n\right)$
denote a fixed ordering of the $n$ units to be tested, represented as an n-tuple. The Ordered Nested Procedure (ONP) proceeds by testing only those units that belong to the current contaminated set, while preserving the original order of the tuple. At each stage, tests are performed on subsets of the remaining contaminated units, maintaining the pre-specified sequence
$\left(u_1,\ldots,u_n\right)$ throughout the procedure.
\end{definition}

\begin{definition}[{Generalized pairwise testing {algorithm (GPTA)}}]
\label{def2}
\,
\begin{itemize}
    \item[\underline{Step 0:}] If the binomial set is empty, terminate the procedure; if it contains exactly one unit, test that unit, add it to the classified set, and terminate; otherwise, proceed to \underline{Step 1}.
    \item[\underline{Step 1:}] Select the first two units from the binomial set and test them together; if the test is negative, add both units to the classified set and return to \underline{Step 0}; if the test is positive, proceed to \underline{Step 2}.
    \item[\underline{Step 2:}] Given the contaminated tuple \((u_{i_1}, u_{i_2})\), test the unit with the smaller contamination probability $p_j$, i.e., the one corresponding to
    \[
        j = \begin{cases}
            i_1, & \text{if } p_{i_1} \leq p_{i_2}, \\
            i_2, & \text{if } p_{i_1} > p_{i_2}.
        \end{cases}
    \]
    If this unit is pure, classify the other unit as contaminated and add both to the classified set; if this unit is defective, add it to the classified set and return the other unit to the binomial set. Then, return to \underline{Step 0}.
\end{itemize}
\end{definition}
Note that the {GPTA} is the ONP when the ordering satisfies \( p_1 \leq \dots \leq p_n \). However, for other orderings, it is in general not an ONP.

The intuition behind Step 2 is clear and has been both explained and formally justified as Result 2 in \cite{M2020}. Moreover, \cite{M2020} proposed the following two conjectures, the first of which was validated through Monte Carlo simulations for population sizes up to 1,000.

\begin{cj}
\label{cj:1}
Let $u_1,u_2\ldots,u_n$ be units ordered so that $p_1\leq p_2\leq \cdots \leq p_n$, with $1-1/\sqrt{2}\leq p_i\leq (3-\sqrt{5})/2$ for all $i$.
Then, among all nested testing procedures that preserve the order $u_1,u_2\ldots,u_n$, the {GPTA} is an optimal nested procedure. It is not necessarily uniquely optimal at the boundary values of the interval.
\end{cj}

\begin{cj}
\label{cj:2}
For every positive integer \( n \) and for every set of defect probabilities \( p_1, \ldots, p_n \) lying in the interval
\[
\bigl[\,1 - \tfrac{1}{\sqrt{2}},\, \tfrac{3 - \sqrt{5}}{2}\bigr],
\]
the {{GPTA}}, when applied to the units \( u_1,\ldots,u_n \) arranged in the (as-yet unknown) optimal order, is an optimal nested procedure.
\end{cj}

\noindent The remark below provides additional insights.
\begin{remark}
For $1 - \frac{1}{\sqrt{2}} \leq p_i \leq \frac{3 - \sqrt{5}}{2}$ for all $i$, the ordered configuration in Conjecture 1 is not optimal, as demonstrated by Example 1 in \cite{M2020}.
Moreover, within the range of $p_i< \frac{3 - \sqrt{5}}{2}$ for all $i$, the ordered configuration is also not optimal for the modified Dorfman and Sterrett \citep{S1957} algorithms, both of which belong to the nested class
as was demonstrated in \cite{M2017}.
\end{remark}

In the following, we establish the validity of the Conjecture \ref{cj:1} and derive a closed-form expression for the expected number of tests required by the {GPTA}. These findings contribute novel perspectives to the theoretical foundations of generalized group testing.

\section{Main Results}
In both statements below, we assume that units $u_1,u_2,\dots,u_n$ are ordered so that their defectiveness probabilities form a non-decreasing sequence:
\begin{equation}\label{e:order}
    p_1\leq p_2\leq\dots\leq p_n.
\end{equation}
\begin{lem}\label{l:GPTP_mean}
Let \( t_{i:n} = t_{i:n}(q_i, \dots, q_n) \), for \( 1 \leq i \leq n \), denote the expected number of tests required by the GPTA on the units \( u_i, \dots, u_n \) preserving ordering given by \eqref{e:order}. Then:
\begin{itemize}
    \item[(i)] \( t_{1:n} = \sum_{i=1}^n \Delta_{i:n} \), where \( \Delta_{n:n} = 1 \) and\footnote{For \( i = n-1 \), the middle sum has an empty index set and is therefore equal to \( 0 \).}
    \begin{equation*}\label{e:delta_expression}
        \Delta_{i:n} = t_{i:n} - t_{i+1:n} = 2 - q_i q_{i+1} + \sum_{j=1}^{n - i - 1} (-1)^j q_i \cdots q_{i + j - 1} \left( 2 - q_{i + j} q_{i + j + 1} \right) + (-1)^{n - i} q_i \cdots q_{n - 1},
    \end{equation*}
    for \( i = 1, \dots, n - 1 \);

    \item[(ii)] For \( n \geq 2 \) and
    \(
    p_i \in \left( 1 - \frac{1}{\sqrt{2}},\, \frac{3 - \sqrt{5}}{2} \right),\,\,\,\, i = 1, \ldots, n,
    \)
    it holds that \( \Delta_{1:n} < 1 \).
\end{itemize}
\end{lem}

\begin{thm}\label{t:main}
Let \(u_1,u_2,\ldots,u_n\) be units ordered so that \eqref{e:order} holds and let
\(
p_i \in \bigl(1 - \frac{1}{\sqrt{2}},\, \frac{3 - \sqrt{5}}{2}\bigr)
\)
for every \(i = 1,\dots,n\).
Then the {GPTA} is the \emph{unique optimal} ONP that preserves this ordering.
\end{thm}

Before we provide the formal proof, the following discussion—presented as comments—may help clarify the current state of knowledge.

\begin{comment}
\begin{enumerate}
\item[(a)]
Conjecture \ref{cj:2} remains open. However, the main part of our proof applies in this case as well: the only missing components are an explicit formula for the expected number of tests and a proof that the optimal procedure cannot begin by testing a single unit.
Nevertheless, the derivations provided in Section~\ref{s:proofs} imply that, under \(
p_i \in \bigl(1-\tfrac{1}{\sqrt{2}},\,\tfrac{3-\sqrt{5}}{2}\bigr)
\), the optimal nested procedure never begins with more than two units.
\item[(b)] The ONPs corresponding to monotone orderings (i.e., $p_1 \leq \dots \leq p_n$ or $p_1 \geq \dots \geq p_n$) generalize the homogeneous case where $p_i = p$ for all $i$. Consequently, our result extends that of \cite{YH1990}, who conjectured that, within a certain range, a homogeneous pairwise testing procedure may be globally optimal. Given that the average number of tests is a continuous function of the $p_i$'s, we conjecture that a similar result may hold in the non-homogeneous case. However, the domain for such optimality is likely to be quite narrow.
\end{enumerate}
\end{comment}

\section{Proofs}\label{s:proofs}
\begin{proof}[Proof of Lemma \ref{l:GPTP_mean}]
By the description of the {GPTA} (see Definition \ref{def2}), and conditioning on the state of nature of the first two units, it follows that for \( n \geq 2 \) (with the convention that \( t_{j:n} = 0 \) for \( j > n \)), we have
\begin{multline}
\label{eq:1}
    t_{1:n}=q_1q_2(1+t_{3:n})+q_1p_2(2+t_{3:n})+p_1(2+t_{2:n})=
    2-q_1q_2+p_1t_{2:n}+q_1t_{3:n}.
\end{multline}

From \eqref{eq:1}, we obtain
\begin{equation}
\label{eq:2}
    \Delta_{1:n}:=t_{1:n}-t_{2:n}=2-q_1q_2-q_1(t_{2:n}-t_{3:n}),\,n\geq 2.
\end{equation}

Set
\begin{equation}
\label{eq:def}
\Delta_{i:n}= t_{i:n} - t_{i+1:n}.
\end{equation}

Similarly to \eqref{eq:1} and \eqref{eq:2}, we obtain for \( {i \geq 1} \) and ${n\geq i+1}$,
\begin{equation}
\label{eq:3}
    \Delta_{i:n}=2-q_iq_{i+1}-q_i(t_{i+1:n}-t_{i+2:n}).
\end{equation}

Recursively appealing to~\eqref{eq:3}, and noting that \( \Delta_{n:n} = 1 \) by definition \eqref{eq:def}, we obtain for \( n \geq 2 \) and $i=1,
\ldots,n-1$,
\begin{multline}\label{e:dr}
    \Delta_{i:n}
    = 2 - q_i q_{i+1} - q_i \Delta_{i+1:n}= \dots \\[6pt]
    = 2 - q_i q_{i+1} + {\sum_{j=1}^{n - i - 1} (-1)^j q_i \cdots q_{i + j - 1} \left( 2 - q_{i + j} q_{i + j + 1} \right) + (-1)^{n - i} q_i \cdots q_{n - 1}}.
\end{multline}

This completes the proof of part~(i) by combining \eqref{eq:def} and \eqref{e:dr}.

It remains to prove (ii). Recall that we assumed that
\begin{align}
\label{eq:A1}
q_1 \geq q_2 \geq \cdots \geq q_n,
\end{align}
and
\begin{align}
\label{eq:A2}
q_i \in \left( \frac{\sqrt{5}-1}{2}, \frac{1}{\sqrt{2}} \right), \quad i = 1, \ldots, n.
\end{align}

First, observe that by \eqref{e:dr},
\begin{equation}
\label{eq:i1}
\Delta_{i:i+1} = 2 - q_i q_{i+1} - q_i < 1,\,\,\,\,i=1,\ldots,n-1,
\end{equation}
where the inequality follows from the assumption in \eqref{eq:A2}.

Also, by \eqref{e:dr},
\begin{align}
\label{eq:i2}
\Delta_{i:i+2} &= 2 - q_i (2 - q_{i+1} q_{i+2})=2 - 2q_i +q_{i}q_{i+1} q_{i+2}\nonumber\\
&
\leq 2 - 2q_i + q_i^3 < 1,\,\,\,\,i=1,\ldots,n-2,
\end{align}
where the first inequality follows from \eqref{eq:A1}, and the second follows from the fact that the function \( f(x) = 2 - 2x + x^3 \) is continuous and strictly decreasing on the interval \eqref{eq:A2}, with \( f\left( \frac{\sqrt{5}-1}{2} \right) = 1 \).

The inequalities \( \Delta_{1:2} < 1 \) and \( \Delta_{1:3} < 1 \) are particular cases of \eqref{eq:i1} and \eqref{eq:i2}, respectively.
We complete the proof by induction. From \eqref{e:dr}, we obtain
\begin{equation}
\label{eq:i3}
    \Delta_{1:m} = \Delta_{1:m-1} + q_1 \cdots q_{m-2} \, (-1)^{m-2} \left( 1 - q_{m-1} - q_{m-1} q_m \right), \quad 3 \leq m \leq n.
\end{equation}
Therefore, for even \( m \), with \( 4 \leq m \leq n \), it follows from \eqref{eq:i3} that
\begin{equation*}
    \Delta_{1:m} = \Delta_{1:m-1} + q_1 \cdots q_{m-2} \left( 1 - q_{m-1} - q_{m-1} q_m \right) < \Delta_{1:m-1} < 1,
\end{equation*}
where the first inequality follows from \eqref{eq:i1}, and the second from the induction assumption.

For odd \( m \), with \( 3 \leq m \leq n \), we obtain from \eqref{eq:i3}
\begin{equation*}
\label{eq:i4}
    \Delta_{1:m} = \Delta_{1:m-2} + q_1 \cdots q_{m-3} \left( 1 - 2q_{m-2} + q_{m-2} q_{m-1} q_m \right) < \Delta_{1:m-2} < 1,
\end{equation*}
where the first inequality follows from \eqref{eq:i2}, and the second from the induction assumption.

\end{proof}

To proceed further, we require two technical lemmas. Their proofs rely on basic properties of optimal alphabetic trees (OATs) introduced in \cite{HT1971}. A clear exposition of the construction of OATs, along with references to related literature, can be found in \cite{D1998} (see also \cite{HS2002}).

\begin{lem}\label{l:single_defective_proc_as_OAT}
Assume that the initial group $(u_1, \dots, u_n)$ is contaminated. If, at each testing stage, a nested procedure is allowed to test only subsets consisting of units from the contaminated set, preserving their order from left to right, then the optimal procedure for isolating one defective unit is given by an OAT corresponding to the sequence of weights
\[
w_1 = \frac{p_1}{\alpha}, \quad w_2 = \frac{p_2 q_1}{\alpha}, \quad \dots, \quad w_n = \frac{p_n q_1 \cdots q_{n-1}}{\alpha}, \quad \text{where } \alpha = 1 - q_1 \cdots q_n.
\]
The isolated unit is the first defective encountered in the sequence $u_1, u_2, \dots, u_n$ when inspected from left to right.
\end{lem}

\begin{proof}
Under the imposed constraints, the process of isolating a single defective unit can be modeled as an extended binary tree with the following structure:
\begin{itemize}
    \item Internal nodes represent tests applied to subsets formed from the leftmost portion of the remaining contaminated set.
    \item The level of each node (i.e., the number of edges from the root) reflects the number of tests conducted up to that stage.
    \item A left branch from an internal node denotes a positive test outcome (i.e., contamination is detected), and a right branch denotes a negative outcome (i.e., pure).
    \item Each terminal node is labeled by a single unit, representing the isolated defective.
\end{itemize}

An example of such a testing tree for a group of five units is shown in Figure~\ref{fig:testing_tree_example2}.

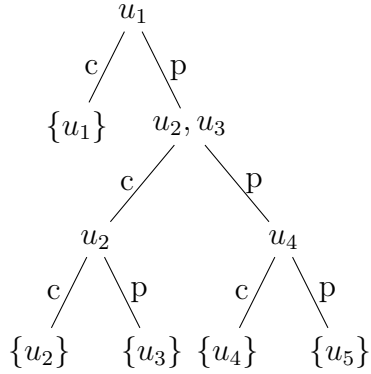
\begin{figure}[ht!]
    \centering
    \begin{tikzpicture}[inner sep=3pt, level distance=1.5cm,
        level 1/.style={sibling distance=1.5cm},
        level 2/.style={sibling distance=2.5cm},
        level 3/.style={sibling distance=1.5cm},
        level 4/.style={sibling distance=3cm},
        level 5/.style={sibling distance=1.5cm}]
        \node {$u_1$}
            child {node {$\{u_1\}$}
                edge from parent
                node[left] {c}
            }
            child {node {$u_2, u_3$}
                child {node {$u_2$}
                    child {node {$\{u_2\}$}
                        edge from parent node[left] {c}
                    }
                    child {node {$\{u_3\}$}
                        edge from parent node[right] {p}
                    }
                    edge from parent node[left] {c}
                }
                child {node {$u_4$}
                    child {node {$\{u_4\}$}
                        edge from parent node[left] {c}
                    }
                    child {node {$\{u_5\}$}
                        edge from parent node[right] {p}
                    }
                    edge from parent node[right] {p}
                }
                edge from parent node[right] {p}
            };
    \end{tikzpicture}
    \caption{Testing tree for the identification of the first defective unit in the contaminated 5-tuple
$(u_1, u_2, u_3, u_4, u_5)$.
Internal nodes (unbraced labels) represent tested subgroups, preserving the original unit order. Edge labels ‘c’ and ‘p’ correspond to contaminated and pure test outcomes, respectively.
Terminal nodes (labels in curly braces) indicate the first defective unit identified by the procedure. The process starts at the root and proceeds by testing subsets at each node, with the choice of the next test determined by the outcome of the previous one.
    }
    \label{fig:testing_tree_example2}
\end{figure}

Due to the ordered nature of the testing procedure, the unit identified by the algorithm is guaranteed to be the first defective encountered in the sequence \( u_1, u_2, \dots, u_n \). Consequently, the path from the root to the terminal node labeled \( \{u_i\} \) is traversed if and only if \( u_i \) is the first defective unit and all preceding units \( u_1, \dots, u_{i-1} \) are pure. Since the length of the path, denoted by \(\ell_i\), corresponds to the number of tests performed by the procedure until termination, the sum
\begin{equation*}
    \sum_{i=1}^n\ell_i \textsf{P}(u_i\text{ defective, }u_1,\dots,u_{i-1}\text{ pure})=
    \alpha^{-1}\sum_{i=1}^n\ell_ip_i\prod_{j=1}^{i-1}q_j
\end{equation*}
represents the total expected number of tests utilized by the procedure.
On the other hand, this sum is also equal to the cost of a binary tree with weights
\[
\alpha^{-1} p_1, \quad \alpha^{-1} p_2 q_1, \quad \dots, \quad \alpha^{-1} p_n q_1 q_2 \cdots q_{n-1}.
\]
Since each such tree is constrained to preserve the order of its terminal nodes—so that the \(i\)-th node (counting from left to right) is labeled \(\{u_i\}\)—it is an alphabetic tree. Therefore, to obtain the optimal procedure, one needs to construct an \emph{Optimal Alphabetic Tree (OAT)}.
\end{proof}

\begin{remark}
Note that the statement of the lemma applies to any ordering of the initial set, provided this initial order is preserved in the terminal nodes. \cite{GH1974} demonstrated that, in the unordered setting, the optimal procedure for isolating a single defective unit is given by the same OAT. However, one must first arrange the probabilities of defectiveness so that they form a non-increasing sequence \(p_1 \geq p_2 \geq \dots \geq p_n.\)
In our case, the order is fixed, which is the key difference.
\end{remark}

\begin{lem}\label{l:OAT_structure}
Consider the optimal procedure described in Lemma \ref{l:single_defective_proc_as_OAT}. Assume that
\begin{align}
\label{eq:con1}
1 - \frac{1}{\sqrt{2}} < p_1 \leq p_2 \leq \dots \leq p_n < \frac{3 - \sqrt{5}}{2}.
\end{align}
Then, for \( n \geq 3 \), in the first testing step, the procedure tests at most the first \( n - 2 \) units, and the penultimate unit is tested individually.
\end{lem}

\begin{proof}
Consider the last three terminal nodes of the OAT described in Lemma~\ref{l:OAT_structure}, which implement the optimal procedure. Their weights are given by
\[
w_{n-i} = \alpha^{-1} p_{n-i} q_1 \cdots q_{n-i-1}, \quad i = 0, 1, 2.
\]
We first show that
\begin{equation}
\label{eq:in1}
w_n < w_{n-2}.
\end{equation}

By simple rearrangement it follows that
\begin{align*}\label{e:condition_on_last_weights}
&
w_n<w_{n-2}\Longleftrightarrow
    q_{n-2}(1+q_{n-1}(1-q_n))<1.
\end{align*}

Using assumption~\eqref{eq:con1}, we immediately obtain
\begin{align*}
q_{n-2}\left(1 + q_{n-1}(1 - q_n)\right)
&< \frac{1}{\sqrt{2}} \left(1 + \frac{1}{\sqrt{2}}\left(1 - \left(1 - \frac{3 - \sqrt{5}}{2}\right)\right)\right)=0.89\cdots
\end{align*}
thus proving~\eqref{eq:in1}.

Now, let us consider the construction process of the OAT. Due to inequality~\eqref{eq:in1}, we have that, during the execution of the Hu--Tucker algorithm, the two nodes corresponding to weights \( w_n \) and \( w_{n-1} \) will be merged before \( w_{n-2} \) is involved in any merging operation. This is a key observation: it guarantees that there exists a step in the algorithm at which \( w_n \) and \( w_{n-1} \)---the two smallest and rightmost weights---are selected for merging.

As a consequence of this merging, both \( w_{n-1} \) and \( w_n \) will be assigned the same level \( \ell \geq 1 \) in the \emph{level assignment phase} of the Hu--Tucker algorithm. This level indicates their distance from the root in the final binary tree structure.

Since the nodes are ordered alphabetically (i.e., left to right from \( w_1 \) to \( w_n \)), the sequence of level assignments produced will be of the form
\[
\ell_1, \ell_2, \dots, \ell_{n-2}, \ell, \ell,
\]
where the last two equal entries \( \ell \) correspond to the merged weights \( w_{n-1} \) and \( w_n \). This level sequence is then used during the \emph{tree reconstruction phase}, in which the final structure of the OAT is obtained by recursively combining nodes according to their assigned levels, while preserving their left-to-right ordering.

Given that the two rightmost entries in the level sequence are equal, the reconstruction algorithm will combine the last two nodes into a common parent. This parent node will therefore appear as the right child of another internal node---specifically, it will be part of the \emph{right subtree of the root}. This follows from the recursive nature of reconstruction {(the rightmost sub-tree is formed from the highest-index elements in the sequence) and the fact that} OATs are always \emph{extended binary trees}\footnote{That is, every internal node, including the root, has exactly two children. {This guarantees that the root has the right subtree.}}. As a result, the \emph{initial test} (i.e., the root-level test in the associated decision process) must distinguish between units located in the left and right subtrees. Because \( w_{n-1} \) and \( w_n \) are located in the right subtree, the initial test must necessarily involve units corresponding to the left subtree---that is, among the units \( w_1, \dots, w_{n-2} \).

This completes the argument and establishes the claim.
\end{proof}

\begin{proof}[Proof of Theorem \ref{t:main}]
Recall that
\begin{equation} \label{eq:prob_bounds}
1 - \frac{1}{\sqrt{2}} < p_1 \leq p_2 \leq \dots \leq p_n < \frac{3 - \sqrt{5}}{2}.
\end{equation}

First, consider \( n = 2 \). By Lemma~\ref{l:GPTP_mean}, we have
\[
t_{1:2} = 2 - q_1 q_2 - q_1 < 1.
\]

Note that \( 2 - q_1 q_2 - q_1 \) is precisely the expected length of an optimal prefix code constructed via Huffman's algorithm \citep{H1952}. Thus, for \( n = 2 \), the group testing procedure coincides with the optimal (i.e., shortest expected length) Huffman code and is therefore optimal.

To finish the proof, we assume by induction that the {GPTA} is the unique optimal procedure for all \( n = 2, \dots, m-1 \), with \( m \geq 3 \). Now, consider the case \( n = m \). We divide the remainder of the proof into two steps:

\begin{itemize}
    \item First, we show that an optimal procedure cannot begin by testing more than two units.
    \item Second, we show that it cannot begin by testing a single unit.
\end{itemize}

Combining these two steps, we conclude that any procedure starting with a test on one unit or on three or more units, and then applying the {GPTA} to the remaining units (which is necessary since the {GPTA} is optimal on sets of size at most \( m-1 \) by the induction hypothesis), cannot be optimal.

Therefore, by induction, the {GPTA} is optimal for all \( n \geq 2 \). The structure of the argument closely follows the approach used by \cite{YH1990}.

\emph{Step 1}. Let $T$ be a procedure that begins by testing $k$ units, where $3 \leq k \leq m$. Define a competing procedure $T'$ as follows:
\begin{itemize}
    \item Initially, $T'$ tests the first $k-1$ units;
    \item If this group is pure, it then tests the $k$-th unit; after that, it proceeds on the remaining $n - k$ units exactly as $T$ does;
    \item If the group of the first $k-1$ units is contaminated, $T'$ deduces that the group of the first $k$ units is also contaminated and continues as $T$ does on the remaining units;
    \item In all cases, $T'$ leverages additional information unavailable to $T$ and optimizes by skipping redundant tests.
\end{itemize}
If the $(k-1)$-st unit is the first defective, then, knowing that the group of the first $k$ units is contaminated, procedure $T$ will apply the optimal nested procedure to identify the first defective. By Lemma~\ref{l:OAT_structure}, it will first classify the first $k-2$ units as pure, and then test the $(k-1)$-st unit individually. Upon receiving a positive result for this test, it will further test the $k$-th unit individually, based on the knowledge that the entire group of the first $k$ units is contaminated.

Thus, in total, $T$ will require the following tests:
\begin{itemize}
    \item one test on the entire group $1{:}k$;
    \item tests to classify the first $k-2$ units;
    \item two individual tests on the $(k-1)$-st and $k$-th units.
\end{itemize}

Procedure $T^\prime$ will skip the test on the $(k-1)$-st unit and will require at most the following tests:
\begin{itemize}
    \item one test on the group $1{:}(k-1)$;
    \item tests to classify the first $k-2$ units;
    \item one test on the $k$-th unit.
\end{itemize}
Thus, in the case where the $(k-1)$-st unit is the first defective, procedure $T'$ saves at least one test compared to $T$. A similar saving occurs if the $k$-th unit is the first defective:
\begin{itemize}
    \item $T$ requires one test on the group $1{:}k$, plus at least one test to classify the units $1{:}(k-2)$, and an individual test on the $(k-1)$-st unit;
    \item $T'$ requires one test on the group $1{:}(k-1)$, and one test on the $k$-th unit.
\end{itemize}
Procedure $T'$ will lose one test compared to $T$ only if the entire group $1{:}k$ is pure. Summing up, we have
\begin{multline*}
    \mathbb{E}[T_{\text{tests}}] - \mathbb{E}[T'_{\text{tests}}] \geq q_1 q_2 \cdots q_{k-2} p_{k-1} + q_1 q_2 \cdots q_{k-1} p_k - q_1 q_2 \cdots q_k = \\
    q_1 q_2 \cdots q_{k-2} \left( p_{k-1} + q_{k-1} p_k - q_{k-1} q_k \right) = \\
    q_1 q_2 \cdots q_{k-2} \left( 1 - 2 q_{k-1} q_k \right) > q_1 q_2 \cdots q_{k-2} \left( 1 - 2 \left( \frac{1}{\sqrt{2}} \right)^2 \right) = 0,
\end{multline*}
since for all $i$, $q_i < \frac{1}{\sqrt{2}}$.

Hence, the optimal procedure cannot start by testing more than two units.

\emph{Step 2}. In this step, we prove that optimal procedure can not start by testing single unit. Let $T$ be such that it first tests one unit and after that applies the {GPTA} on the remaining units (since the latter is optimal by inductive assumption). One needs to demonstrate that
\begin{equation*}
    E(T_{tests})=1+{t_{2:n}(q_2,\dots,q_{n})>
    t_{1:n}(q_1,\dots,q_{n})}.
\end{equation*}
The latter is equivalent to
\begin{equation*}
    1>{t_{1:n}(q_1,\dots,q_{n})-t_{2:n}(q_2,\dots,q_{n})}
\end{equation*}
and follows immediately from Lemma \ref{l:GPTP_mean} (ii).

\end{proof}

{}


\begin{thebibliography}{99}
\bibitem[\protect\citeauthoryear{ Aldridge \it{et~al.}}{2019}]{AJS2019}
Aldridge, M.,  Johnson, O., Scarlett, J. (2019).
\newblock Group Testing: An Information Theory Perspective.
\newblock {\emph   Found. Trend. Comms. Inf. Theory} {\textbf{15},} 196--392.




\bibitem[\protect\citeauthoryear{ Barg and Mazumdar}{2017}]{BM2017}
Barg, A., Mazumdar, A. (2017).
\newblock Group testing schemes from codes and designs.
\newblock {\emph   IEEE Trans. Inform. Theory} {\textbf{11},} 7131--7141.





\bibitem[\protect\citeauthoryear{\v{C}i\v{z}ikovien\.{e} and Skorniakov}{2024}]{CS2024}
 \v{C}i\v{z}ikovien\.{e}, U., Skorniakov, V. (2024).
\newblock On the optimal pairwise group testing algorithm.
\newblock {\it Braz. J. Probab. Stat.} {\textbf{38},} 253--265.







\bibitem[\protect\citeauthoryear{Davis}{1998}]{D1998}
Davis, S. (1998).
\newblock Hu-Tucker algorithm for building optimal alphabetic binary search
  trees.
{\it  Master's thesis, Rochester Institute of Technology.}



\bibitem[\protect\citeauthoryear{Dorfman}{1943}]{D1943}
Dorfman, R. (1943).
\newblock The detection of defective members of large populations.
\newblock {\emph The Annals of Mathematical Statistics } {\textbf{14},} 436--440.

\bibitem[\protect\citeauthoryear{Du and Hwang}{1999}]{Dh1999}
Du, D., Hwang, F. K. (1999).
\newblock Combinatorial Group Testing and its Applications. {\it World
Scientific, Singapore}.

\bibitem[\protect\citeauthoryear{Du and Hwang}{2006}]{Dh2006}
Du, D., Hwang, F. K. (2006).
\newblock Pooling Design and Nonadaptive Group Testing: Important Tools for DNA Sequencing. {\it World
Scientific, Singapore}.



\bibitem[\protect\citeauthoryear{Garey and Hwang}{1974}]{GH1974}
Garey, M.~R., Hwang, F.~K. (1974).
\newblock Isolating a single defective using group testing.
\newblock {\emph J. Amer. Statist. Assoc. } {\textbf{69},} 151--153.




\bibitem[\protect\citeauthoryear{Haymaker and O$^{'}$Pella}{2025}]{HO2025}
Haymaker, K., O$^{'}$Pella, J. (2025).
\newblock Coding Theory and Pooled Testing for COVID-19.
\newblock {\emph Mathematics Magazine.} To appear.



\bibitem[\protect\citeauthoryear{Hu and Shing}{2002}]{HS2002}
Hu, T.~C., and Shing, M.~T. (2002).
\newblock Combinatorial algorithms. Enlarged second edition.
{\it Dover Publications, Inc., Mineola, NY}.

\bibitem[\protect\citeauthoryear{Hu and Tucker}{1971}]{HT1971}
Hu, T.~C., and Tucker, A.~C. (1971).
\newblock Optimal computer search trees and variable-length alphabetical codes.
\newblock {\emph  SIAM J. Appl. Math. } {\textbf{21},} 514--532.





\bibitem[\protect\citeauthoryear{Huffman}{1952}]{H1952}
Huffman, D. A. (1952).
\newblock A Method for the Construction of Minimum-Redundancy Codes.
\newblock {\emph Proceedings of the I.R.E. } {\textbf{40},} 1098--1101.


\bibitem[\protect\citeauthoryear{Hwang}{1975}]{H1975}
Hwang, F.~K. (1975).
\newblock A generalized binomial group testing problem.
\newblock {\emph J. Amer. Statist. Assoc.} {\textbf{70},} 923--926.



\bibitem[\protect\citeauthoryear{Katona}{1973}]{K1973}
Katona, G. O. H. (1973).
\newblock
Combinatorial search problems.
\newblock
{\it J.N. Srivastava et al., A Survey of combinatorial Theory,} 285--308.



\bibitem[\protect\citeauthoryear{Kurtz and Sidi}{1988}]{KS1988}
Kurtz, D., and Sidi, M. (1988).
\newblock Multiple access algorithms via group testing for heterogeneous population of users.
\newblock {\emph  IEEE Trans. Commun. } {\textbf{36},} 1316--1323.


\bibitem[\protect\citeauthoryear{Malinovsky}{2019}]{M2017}
Malinovsky, Y. (2019).
\newblock Sterrett procedure for the generalized group testing problem.
\newblock {\it Methodology and Computing in Applied Probability} {\textbf{21},}  829--840.

\bibitem[\protect\citeauthoryear{Malinovsky}{2020}]{M2020}
Malinovsky, Y. (2020).
\newblock Conjectures on optimal nested generalized group testing algorithm.
\newblock {\it Appl. Stoch. Models Bus. Ind.} {\textbf{36},}  1029--1036.


\bibitem[\protect\citeauthoryear{Malinovsky and Albert}{2019}]{MA2019} Malinovsky, Y., Albert, P. S. (2019).
\newblock Revisiting nested group testing procedures: new results, comparisons, and robustness.
\newblock {\it The American Statistician} {\textbf{73},} 117--125.

\bibitem[\protect\citeauthoryear{Malinovsky and Albert}{2021}]{MA2021} Malinovsky, Y., Albert, P. S. (2021).
\newblock Nested Group Testing Procedures for Screening.
\newblock {\it Wiley StatsRef-Statistics Reference Online}.










\bibitem[\protect\citeauthoryear{Nikolopoulos \it{et~al.}}{2023}]{N2023}
Nikolopoulos, P.,  Srinivasavaradhan, S.~R., Guo, T., Fragouli, C., Diggavi, S.~N. (2023).
\newblock Community-aware group testing.
\newblock {\emph   IEEE Trans. Inform. Theory} {\textbf{69(7)},} 4361--4383.




\bibitem[\protect\citeauthoryear{Noonan and Zhigljavsky}{2022}]{NZ2022}
Zhigljavsky, A., Noonan, J. (2022).
\newblock Random and quasi-random designs in group testing.
\newblock {\emph J. Statist. Plann. Inference} {\textbf{ 221},} 29--54.



\bibitem[\protect\citeauthoryear{Papanicolaou}{2021}]{P2021}
Papanicolaou, V.~G. (2021).
\newblock A binary search scheme for determining all contaminated specimens.
\newblock {\emph J. Math. Biol. } {\textbf{83},} 31 pp.


\bibitem[\protect\citeauthoryear{Skorniakov}{2025}]{SV2025}
Skorniakov, V. (2025).
\newblock On the optimal nested ordered algorithm.
\newblock Unpublished manuscript.


\bibitem[\protect\citeauthoryear{Sobel and Groll}{1959}]{SG1959}
Sobel, M., Groll, P. A. (1959).
\newblock Group testing to eliminate efficiently all defectives in a binomial sample.
\newblock {\it Bell System Tech. J.} {\textbf{ 38},} 1179--1252.

\bibitem[\protect\citeauthoryear{Sobel}{1960}]{S1960}
Sobel, M. (1960).
\newblock Group testing to classify efficiently all defectives in a binomial sample.
\newblock {\it Information and Decision Processes (R. E. Machol, ed.; McGraw-Hill, New York),} pp. 127-161.





\bibitem[\protect\citeauthoryear{Sterrett}{1957}]{S1957}
Sterrett, A. (1957).
\newblock On the detection of defective members of large populations.
\newblock {\emph The Annals of Mathematical Statistics } {\textbf{ 28},} 1033--1036.


\bibitem[\protect\citeauthoryear{Ungar}{1960}]{U1960}
Ungar, P. (1960). Cutoff points in group testing.
{\it Comm. Pure Appl. Math.} {\textbf {13},} 49--54.




\bibitem[\protect\citeauthoryear{Yao and Hwang}{1988}]{YH1988b}
Yao, Y. C., Hwang, F. K. (1988).
\newblock  Individual testing of independent items in optimum group testing.
\newblock {\emph Probab. Eng. Inform. Sci.} {\textbf{2},} 23--29.



\bibitem[\protect\citeauthoryear{Yao and Hwang}{1990}]{YH1990}
Yao, Y. C., Hwang, F. K. (1990).
\newblock  On optimal nested group testing algorithms.
\newblock {\emph J. Stat. Plan. Inf.} {\textbf{24},} 167--175.

\end{thebibliography}
\end{document}